\documentclass{amsart}

\newtheorem{thm}{Theorem}[section]
\newtheorem{theorem}[thm]{Theorem}

\newtheorem{lemma}[thm]{Lemma}

\newtheorem{proposition}[thm]{Proposition}

\theoremstyle{definition}

\newtheorem{definition}[thm]{Definition}

\theoremstyle{remark}

\newtheorem{question}{Question}

\newtheorem{claim}{Claim}

\def\val{\operatorname{val}}

\usepackage{amssymb}
\usepackage{amsrefs}

\DefineSimpleKey{bib}{primaryclass}{}
\DefineSimpleKey{bib}{archiveprefix}{}

\BibSpec{arXiv}{%
  +{}{\PrintAuthors}{author}
  +{,}{ \textit}{title}
  +{}{ \parenthesize}{year}
  +{,}{ arXiv }{eprint}
  +{,}{ primary class }{primaryclass}
}

\def\dom{\operatorname{dom}}

\begin{document}
\title{
Some results on the  
$\pi$-weight of countable
Fr\'echet-Urysohn spaces
}
\author{Alan Dow}
\address{Department of Mathematics and Statistics, University
of North Carolina at Charlotte, Charlotte, NC 28223, USA}

\begin{abstract}
The $\pi$-weight  spectrum for countable regular Fr\'echet-Urysohn
   spaces  is the set of uncountable cardinals 
   that are equal to the $\pi$-weight for some such space. 
We determine this $\pi$-weight spectrum in  the 
standard 
 Miller rational perfect set model  and in the  Random real model.
 \end{abstract}


\maketitle

\section{Introduction}

A topological space is Fr\'echet-Urysohn providing it satisfies
that if a point is in the closure of a set, there is an
$\omega$-sequence
from that set converging to the point. In this paper we
assume that all topologies are regular and Hausdorff.
A $\pi$-base for a topology
is a family of non-empty open sets such that every non-empty
member of the topology contains one. The $\pi$-weight of a topological
space is the minimum cardinality of a $\pi$-base. Determining the possible
values of the $\pi$-weights of countable Fr\'echet-Urysohn spaces,
 i.e. the $\pi$-weight spectrum, seems to be an interesting challenge.
 Malyhin's problem for countable Fr\'echet-Urysohn
 topological groups motivated
 Juhasz's  question  of whether the existence of a countable 
 Fr\'echet-Urysohn space with uncountable $\pi$-weight could
 be proven in ZFC.  Naturally, for every cardinal $\kappa <\mathfrak p$,
  every countable dense subset of the Tychonoff product $2^\kappa$ 
  is a countable Fr\'echet-Urysohn space with $\pi$-weight equal to $\kappa$. 
 The other general known results 
 concerning the possible uncountable $\pi$-weights of countable 
 Fr\'echet-Urysohn spaces  seem to be these three. 
 There is a space with $\pi$-weight at least $\mathfrak b$
  \cite{DowFrechet}, there are models in which spaces 
  may have  $\pi$-weight $\mathfrak b = \mathfrak c>\omega_1$
  but none with $\pi$-weight strictly between $\omega$ and
   $\mathfrak c$  \cite{Hrusak-Ramos}, in the Cohen model
   there are no spaces with $\pi$-weight greater than $\aleph_1$.
   It was also shown in \cite{Hrusak-Ramos} that in many models,
   in fact 
   those satisfying a weak parametrized $\diamondsuit$ principle,
   there are spaces with $\pi$-weight $\aleph_1$. 

   One natural source of examples of Fr\'echet-Urysohn spaces
   is $C_p(X)$ for a space $X\subset \mathbb R$ 
   that is a $\gamma$-set (see \cite{Gerlits} and \cite{haberMiller}). 
In particular countable dense subsets of such   $C_p(X)$ spaces
yield countable Fr\'echet-Urysohn spaces with $\pi$-weight
equal to the cardinality of the $\gamma$-set. 
    It is well-known that, in Laver's model of the Borel conjecture,
     there are no uncountable $\gamma$-sets and so these
     were not a ZFC source of countable Fr\'echet-Urysohn spaces
     of uncountable $\pi$-weight. One of the very interesting
     results in \cite{haberMiller}*{Theorem 1.3} is that in the usual Miller model,
     namely the forcing extension by the countable support iteration of the Miller trees poset, 
     there are no $\gamma$-sets of cardinality greater
     than $\aleph_1$. The same result holds  
     in the random real model (adding any number of random
     reals) as can be deduced 
     from   \cite{Goldstern}*{Lemma 2.29} as explained
     in   Proposition \ref{randomgamma}.   

     The main results of this paper are to determine the very different
     $\pi$-weight
     spectrum of countable Fr\'echet-Urysohn spaces in the Miller model
     and in the $\kappa$ random real model for any cardinal $\kappa>\omega_1$.
     As is well-known, the Miller model
     satisfies that $\omega_1 = \mathfrak b < \mathfrak d = \mathfrak c$,
     while $\omega_1 = \mathfrak b = \mathfrak d < \kappa = \mathfrak c$
     holds in the random real models. In the Miller model, like the Cohen model, 
     every countable regular Fr\'echet-Urysohn space has $\pi$-weight
     at most $\aleph_1$. In any model obtained  by adding $\kappa$ random
     reals over a model of CH, every cardinal $\lambda$ with
      $\omega \leq \lambda \leq \kappa$ equals 
      the $\pi$-weight of some countable regular Fr\'echet-Urysohn space.

\section{Fr\'echet-Urysohn spaces in the Miller model}

 In this section we prove this next theorem.   The proof is completed
 at the end of the section.

 \begin{theorem}[CH] In the forcing extension by\label{mainMiller} a countable
 support iteration of length $\omega_2$ of Miller's rational perfect set poset,
  every regular Fr\'echet-Urysohn countable space has
   $\pi$-weight at most $\aleph_1$.
\end{theorem}

By a Miller 
tree we mean a sub-tree  $T$ of $\omega^{<\omega}$ 
which satisfies that every branching node
of $T$ has infinitely many immediate successors
 and every $t\in T$ has an extension that is a
  branching node.  We let $\mathbb M$ denote the
  forcing notion consisting of the set
  of Miller trees ordered by 
   $T_2 < T_1$, in the forcing sense, if $T_2\subset T_1$.

The main new idea of the proof is the preservation of the
following bounding notion. 

\begin{definition} Say that  $(\tau,\mathcal I)$ 
is a Fr\'echet-Urysohn pair if 
 $\tau$ is a clopen basis for a topology on $\omega$
 and 
 $\mathcal I$ is a family of $\tau$-converging sequences.

Say that a subset $Y$ of $\omega$ is $\mathcal I$-closed
if $Y$ contains the $\tau$-limit of $I$ for every
$I\in\mathcal I$ that meets $Y$ in an infinite set.
Say that a  set $Y$ is $(\tau , \mathcal I)$-nowhere dense
if $Y$ is $\mathcal I$-closed and 
$U\setminus Y$ is not empty for all $\emptyset\neq U\in \tau$. 

 For a countable elementary submodel $M$ of $H(\mathfrak c^+)$,
  say that a set $A$ is $(M,\tau,\mathcal I)$-bounding 
  if $A\setminus Y$ is infinite for every subset $Y$ of $\omega$
  in $M$ that is $(\tau, \mathcal I)$-nowhere dense.
\end{definition}

 Let us recall that if $G$ is a generic filter for some
 poset and if $\dot Y$ is a $P$-name for a set such
 that $\val_G(\dot Y)$ is a $(\tau, \mathcal I )$-nowhere dense
 set for some Fr\'echet-Urysohn pair $(\tau , \mathcal I)$,
 then there is a $P$-name $\dot Y_1$ satisfying that
  $1_P$ forces that  $\dot Y_1$ is $(\tau, \mathcal I)$-nowhere
  dense and that $\val_G(\dot Y_1) = \val_G(\dot Y)$. 
  Therefore in what follows we simply focus on names that
   are forced by $1$  to be $(\tau, \mathcal I)$-nowhere dense.

\medskip

Our proof of the  theorem will follow Shelah's methods for
proving that countable support iterations of proper almost 
$\omega^\omega$-bounding posets are weakly $\omega^\omega$-bounding. 
More precisely we will utilize the presentation by Abraham,
 see \cite{Abraham}*{4.1}. Unfortunately we have been unable
 to formulate a suitable relation $R$ so that we could just
  quote one of Shelah's many preservation results.
  Nor could we construct a proof based on the Laver method
  of proving that single stage names of reals can give much information
  about general names of reals as in \cite{RZmsep}*{Lemma 2.2}.
  
  For a poset
   $P$, we will use $\Gamma_P$ to denote the canonical name
   for the $P$-generic filter. Recall that if $P\in M$
   where $M$
   is a countable elementary submodel of $H(\theta )$,
   for a suitably large $\theta$, 
   and if $q$ is an $(M,P)$-generic condition 
   then  $q$ forces that  $M[\Gamma_P]$ is a countable
   elementary submodel of $H(\theta)[\Gamma_P] $.
   A subset $S$ of $P$ is said to be pre-dense below
    a condition $p\in P$, if for every $q\leq p$,
     there is an $s\in S$ such that $q$ and $s$ are
     compatible in $P$. Given two conditions $r,p$ of $P$
     let us say that a set $S$ is pre-dense below
     $r\wedge p$ if $S$ is pre-dense below 
     every common extension of $r$ and $p$.

   \medskip

  Naturally we begin with the single stage poset.
  The  
 minimal branching node of $T\in \mathbb{M}$ is denoted as 
$\mathop{stem}(T)$. For each $n\in\omega$,
  $\mathbf{Br}_n(T)$ is the set of branching
  nodes $t$ of $T$ that satisfy the
  set of $s < t$ that are branching nodes
  has cardinality $n$. 
   For  $T_1, T_2\in \mathbb{M}$ and $n\in\omega$,  
   the  relation  $T_2 <_n T_1$ corresponds to
    $T_2 \leq T_1$ and $\mathbf{Br}_n (T_2) = \mathbf{Br}_n(T_1)$.

For $T\in \mathbb{M}$ and $t\in T$, the sub-tree
 $T_t = \{ s\in T : s\leq t \ \mbox{or}\ t\leq s\}$ 
 satisfies that $T_t \leq T$.   For an $n\in\omega$
 and $T_n\in\mathbb{M}$, a standard fusion step
 to construct $T_{n+1} <_n T_n$ is to choose,
 for each $t\in \mathbf{Br}_{n+1}(T_n)$ any condition
  $\tilde T_t < (T_n)_t$ and to let
   $T_{n+1} = \bigcup \{ \tilde T_t : t\in \mathbf{Br}_{n+1}(T_n)\}$. 
Indeed, if $T_0\in \mathbb{M}$ is an element of a countable 
elementary submodel $M$ of $H(\mathfrak c^+)$ and if 
 $\{D_n : n\in \omega\}$ is an enumeration of the dense subsets
 of $\mathbb{M}$ that are elements of $M$, 
  then $ T_\omega = \bigcap \{ T_n : n\in\omega\}$ is an
   $(M,\mathbb{M})$-generic condition if, for
   each $n\in\omega$, $T_{n+1} <_n T_n$ satisfies
   that, $(T_{n+1})_t\in D_n\cap M$ for each 
    $t\in \mathbf{Br}_{n+1}(T_n)$. Notice that it is not
    necessary that $T_{n+1}\in M$, only
    that $(T_{n+1})_t\in M$ for each $
    t\in 
    \mathbf{Br}_{n+1}(T_n)$.

\begin{lemma}  Let $(\tau,\mathcal I)$ be a Fr\'echet-Urysohn pair
and let $A$ be an $(M,\tau,\mathcal I)$-bounding set\label{singlestep} for
a countable elementary submodel $M$ of $H(\mathfrak c^+)$ 
such that $\tau,\mathcal I\in M$.  Then for every
 $\mathbb{M}$ condition $T_0\in M$, 
 there is an $(M,\mathbb{M})$-generic condition 
  $T_\omega $ extending $T$ that    forces that $A$ is
   $(M[ \Gamma_{\mathbb M}],\tau,\mathcal I)$-bounding. 
\end{lemma}

\begin{proof}  Fix an enumeration, $\{ D_n : n\in\omega\}$ of
the dense subsets of $\mathbb{M}$ that are elements of $M$. 
Also, fix an enumeration $\{ \dot Y_n : n\in\omega\}$  of
the $\mathbb{M}$-names in $M$ that are forced by $1$
 to be  subsets of $\omega$ that are
   $(\tau , \mathcal I )$-nowhere dense. 
   For convenience ensure that each $\dot Y\in \{\dot Y_n : n\in\omega\}$
   is   listed infinitely many times.
Let $n\in\omega$ and assume that, we have chosen 
 $\{ T_k : k\leq n\}\subset \mathbb{M}$ so that, for
 each $k < n$
 \begin{enumerate}
 \item $T_{k+1} <_k T_k$,
\item  for each $t\in \mathbf{Br}_{k+1}(T_k)$, 
    $(T_k)_t\in M\cap D_k$,
    \item for each $t\in \mathbf{Br}_{k+1}(T_k)$, there
    is an $a\in A$ such that $(T_{k+1})_t $ forces that
      $a\notin \dot Y_k$. 
 \end{enumerate}
 First we explain   how to construct $T_{n+1}$. Fix any
  $t\in \mathbf{Br}_{n+1}(T_n)$.   Consider $\dot Y_n$,
   and let $(\dot Y_n)^-_{(T_n)_t}$ denote the set of 
   all $m\in \omega$ such that $(T_n)_t$ forces
   that $m\in \dot Y_n$. Clearly $(T_n)_t$ forces
that $(\dot Y_n)^-_{(T_n)_t}$ is a subset of $\dot Y_n$.
Since $1$ forces that $\dot Y_n$ is $\mathcal I$-closed
it follows that $(\dot Y_n)^-_{(T_n)_t}$ is $\mathcal I$-closed
and, similarly it follows that $(\dot Y_n)^-_{(T_n)_t}$ is 
$(\tau , \mathcal I)$-nowhere dense. Since, by the induction
assumption $(T_n)_t$ is an element of $M$, it follows
that $(\dot Y_n)^-_{(T_n)_t}$ is also an element of $M$.
Therefore, there is an $n<a\in A$ such that $(T_n)_t$
does not force that $a\in \dot Y_n$.  Choose any
 $\tilde T_{n, t} < (T_n)_t$ in $M$ that forces
 that $a\notin \dot Y_n$. By further extending, if necessary,
  we can also assume that $\tilde T_{n ,t}$ is an element
  of $D_n\cap M$. Then we set $T_{n+1}$ equal to 
   $\bigcup \{ \tilde T_{n,t} : t\in \mathbf{Br}_{n+1}(T_n)\}$. 
   It should be clear that the inductive hypotheses are preserved.

   Now let $T_\omega = \bigcap T_n$. As discussed above
    $T_\omega$ is an extension of $T_0$ that is an
    $(M,\mathbb{M})$-generic condition.  Suppose
    that $T_\omega$ is an element of a generic filter 
     $G$ for $\mathbb{M}$ and let $\dot Y\in M$
     be any $\mathbb{M}$-name that is forced
     by $1$ to be $(\tau , \mathcal I)$-nowhere dense.
     Then fix any $n\in\omega$ such that $\dot Y_n   =\dot Y$.
The family $\{ \tilde T_{n,t} : t\in \mathbf{Br}_{n+1}(T_n)\}$
 is pre-dense below $T_{n+1}$, and therefore also pre-dense
 below $T_\omega$. Choose the unique $t\in \mathbf{Br}_{n+1}(T_n)$
 such that $\tilde T_{n,t}\in G$. By assumption, 
 there is an $n<a\in A$ such that $a\notin \val_{G}(\dot Y_n)$. 
\end{proof}

 A somewhat subtle point of the above construction is that, almost surely,
   $T_{n+1}$ is not an element of $M$. However, for each $t\in 
    \mathbf{Br}_{n+1}(T_{n+1})$, $(T_{n+1})_t$ is an element of $M$.
  If we consider the standard countable support iteration,
    $\langle P_\lambda, \dot Q_\alpha : \lambda\leq\omega_2, \alpha < \omega_2\rangle$
    where, for all $\alpha <\omega_2$, $P_\alpha\Vdash \dot Q_\alpha = \mathbb M$,
    then, officially we can let $p_{n+1}\in P_1$, be defined by
    the property that $1 = p_{n+1}\restriction 0 $ forces that
     $p_{n+1}(0) = (T_{n+1})_{\dot t}$,  where $\dot t$ is the
     $P_0$-name of the unique element $t$ of $\mathbf{Br}_{n+1}(T_{n+1})$ such
     that $(T_{n+1})_t \in  \Gamma_{\dot Q_0}$.  Then, even though $p_{n+1}$ is not
      an element of $M$, $p_{n+1}\restriction 0$ forces that 
       $p_{n+1}(0)$ \textbf{is} an element of $M[\Gamma_{P_0}] = M$. 
       This is a key element of this next proof that is 
       based on \cite{Abraham}*{4.1}.

\begin{lemma}[CH]  Let $\lambda< \omega_2$ and\label{Milleriterate} 
$\langle P_\beta, \dot Q_\alpha : \beta\leq \lambda, 
 \alpha < \lambda\rangle$ be   the countable support iteration where, for each
  $\alpha < \lambda$, $\dot Q_\alpha$ is the $P_\alpha$-name for the Miller tree poset $\mathbb M$.
  Let $(\tau , \mathcal I)$ be a Fr\'echt-Urysohn pair and let 
  $\langle P_\beta, \dot Q_\alpha : \beta\leq \lambda, 
 \alpha < \lambda\rangle$  and $(\tau, \mathcal I)$ be elements of a countable
 elementary submodel $M$ of $H(\mathfrak c^+)$. 
Let $A$ be $(M,\tau,\mathcal I )$-bounding.

Then,  for
 any  $\gamma_0 < \lambda$ and $q_0\in P_{\gamma_0}$ that is
  $(M,P_{\gamma_0})$-generic and that forces that $A$ is
   $(M[\Gamma_{P_{\gamma_0}}],\tau,\mathcal I)$-bounding and
   any $P_{\gamma_0}$-name $\dot p_0$ that is forced by
    $q_0$ to be an element of $M\cap P_\lambda$
and also $q_0\Vdash \dot p_0\restriction \gamma_0\in \Gamma_{P_{\gamma_0}}$,  
  then    there is a $q<p_0$ that is $(M,P_\lambda)$-generic such
  that $q\restriction \gamma_0 = q_0$ and such
  that $q$ forces that $A$ remains $(M[\Gamma_{P_\lambda}],\tau,\mathcal I)$-bounding.
\end{lemma}

\begin{proof} We prove the Lemma by induction on $\lambda$.
Let $\{ \dot Y_n : n\in\omega\}$ enumerate all the 
$P_\lambda$-names in $M$ such that $1_{P_\lambda}$ forces
that $\dot Y_n$ is $(\tau,\mathcal I)$-nowhere dense. 
Ensure that each member of $\{ \dot Y_n : n\in \omega\}$
is enumerated infinitely many times. Also let $\{ D_n : n\in\omega\}$
be an enumeration of the dense subsets of $P_\lambda$ that
 are elements of $M$.

It is a standard
result of proper forcing, combined with Lemma \ref{singlestep}, that we
may assume that $\lambda$ is a limit ordinal.  Let $\delta$ be the supremum
of $M\cap \lambda$.  Choose any strictly increasing sequence
 $\langle\gamma_n : n\in \omega\rangle$ of ordinals
 in $M\cap \lambda$ that is cofinal in $\delta$. Following
  \cite{Abraham}*{Theorem 4.1}, we define, by induction,
  conditions $q_n\in P_{\gamma_n}$ that are $(M,P_{\gamma_n})$-generic,
  and $P_{\gamma_n}$-names  $\dot p_n$ such that:
  \begin{enumerate}
\item $q_{n+1}\restriction \gamma_n = q_n$,
\item $q_n$ forces that $\dot p_n$ is in $P_\lambda\cap M$  and extends
 $\dot p_{n-1}$, and $q_n$ also forces that $\dot p_n\restriction \gamma_n$
  is in $\Gamma_{P_{\gamma_n}}$,
  \item $q_n$ forces that $A$ is $(M[\Gamma_{P_{\gamma_n}}],\tau, \mathcal I)$-bounding,
 \item for $n>1$,  $q_n$ forces\label{four} that if $p_n\in P_\lambda\cap M$ 
  is equal to $\dot p_n$, then $p_n\in D_n$ and there is an $n<a\in A$ such that
     $p_n\Vdash a\notin \dot Y_n$.
  \end{enumerate}

  If the recursive construction succeeds, then, as per the proof
  of \cite{Abraham}*{Lemma 2.8}, the condition 
   $q = \bigcup_n q_n$ is  $(M,P_\lambda)$-generic
   and $q< p_0$. Moreover if $q$ is an element of
    a $P_\lambda$-generic filter $G$, 
    then,  for each $n\in\omega$,
      there is a $p_n\in P_\lambda\cap M$
     such that $p_n = \val_{G}(\dot p_n)$,
     and, by item \ref{four}, there is an
      $n < a \in A\setminus \val_{G}(\dot Y_n)$.
   Therefore, if $q$ is an element of a $P_\lambda$-generic
   filter $  G$,  we will have that $A$ is $(M[G],\tau,\mathcal I)$-bounding
   as required.

   Assume that $0<n\in\omega$ and that $q_{k}$ and $\dot p_k$ have
   been chosen for all $k<n$ and that $q_{n-1}$ is
    $(M,P_{\gamma_{n-1}})$-generic and 
   forces
   that $A$ is  $(M[\Gamma_{P_{\gamma_{n-1}}}], \tau,\mathcal I)$-bounding.
   Let $q_{n-1}$ be an element of  a $P_{\gamma_{n-1}}$-generic filter
    $G_{\gamma_{n-1}}$ and let $p_{n-1} \in M$ be the valuation of 
     $\dot p_{n-1}$ by $G_{\gamma_{n-1}}$.  
Consider the
      $P_\lambda$-name $\dot Y_n$ in $M[G_{\gamma_{n-1}}]$ 
      and let $Y_n^-$ be the set of all integers
       $j$ such that,
       for some $r\in G_{\gamma_{n-1}}$,
       the set of conditions that 
       force that $j$ is in $\dot Y_n$
       is predense below 
        $r\wedge p_{n}$ . 
Since $p_0$ forces that $\dot Y_n$ is $(\tau,\mathcal I)$-nowhere
dense and $p_n<p_0$ forces that $Y_n^-$ is a subset of 
 $\dot Y_n$, it follows that $Y_n^-$ is also
  $(\tau,\mathcal I)$-nowhere dense. It is evident
  that $Y_n^-$ does not contain any non-empty $U\in \tau$,
   and, by the definition of the forcing relation,
    if $Y_n^-\cap I$ is infinite for some $I\in \mathcal I$,
    then $p_{n}$ forces, over $V[G_{\gamma_{n-1}}]$,
    that the $\tau$-limit of $I$ is 
    in $\dot Y_n$. Therefore, we have that $A\setminus Y_n^-$ 
    is an infinite set. Choose any $n < a\in A\setminus Y_n^-$
    and note that there is a condition $\tilde p_n < p_{n-1} $ in $M\cap P_\lambda$
    such that  $\tilde p_n\restriction \gamma_{n-1}\in G_{\gamma_{n-1}}$
    and 
    $\tilde p_n \Vdash a\notin \dot Y_n$. 
    By further extending $\tilde p_n$ we may assume that it is
    also
    an element of $D_n\cap M$. Then, as in \cite{Abraham}*{2.8},
return to the ground model, and let $\dot p_n$ be a
 $P_{\gamma_{n-1}}$-name that is forced by $q_n$ to equal 
 such a condition $\tilde p_n\in M$ with the above properties.
 Since $\gamma_n<\lambda$, we can apply the induction hypothesis
 to secure our condition $q_n$ that is $(M,P_{\gamma_n})$-generic
 such that   $q_n\restriction \gamma_{n-1} = q_{n-1}$ and 
 such that $q_n$ 
 forces that $\dot p_n\restriction \gamma_n\in \Gamma_{P_{\gamma_n}}$
 and that $A$ is  $(M[\Gamma_{P_{\gamma_n}}],\tau,\mathcal I)$-bounding.
\end{proof}

Now we prove our main result about Miller forcing and 
countable Fr\'echet-Urysohn spaces. 

\bgroup

\def\proofname{Proof of Theorem \ref{mainMiller}}

\begin{proof}  We are working in a ground model of CH.
Let $\langle P_\beta , \dot Q_\alpha : \beta\leq \omega_2, \alpha < \omega_2\rangle$
be the countable support iteration in which, for every $\alpha<\omega_2$,
 $\dot Q_\alpha$ is the $P_\alpha$-name for $\mathbb M$. 
 Let $\{ \dot U_\xi : \xi < \omega_2\} $ be a list of
 $P_{\omega_2}$-names for non-empty subsets of $\omega$ such that $1$
 forces that the family forms a clopen basis for a Hausdorff
 Fr\'echet-Urysohn topology on $\omega$. Since $P_{\omega_2}$ is
 proper, has a dense subset of cardinality $\aleph_2$,  and 
 satisfies the $\aleph_2$-chain condition every regular Hausdorff
 topology on
  $\omega$ can be assumed to have such a basis. 

  Choose any elementary submodel $M$ of $H(\aleph_3)$ satisfying
  that $\{\dot U_\xi : \xi\in\omega_2\}\in M$, $M^\omega\subset M$,
  and such that $M$ has cardinality $\aleph_1$. It follows
  that $\mu = M\cap \omega_2$ is an element of $\omega_2$.
  Let $G_\mu$ be any $P_\mu$-generic filter and 
  let $\tau = \{ \val_{G_\mu}(\dot U_\xi) : \xi\in \mu\}$.
  We prove that $1_{P_{\omega_2}}$ forces  that
    $\{ \dot U_\xi : \xi <\mu\}$ is a $\pi$-basis for  the
   final topology  given by
     $\{ \dot U_\xi : \xi < \omega_2\}$. 
     Since $P_{\omega_2}$ is proper and satisfies
     that $\aleph_2$-cc, it follows
     that $M[G_\mu]$, like $M$,
     is closed under $\omega$-sequences. Additionally,
      $M[\Gamma_{P_{\omega_2}}]$ is forced by $1$
      to be an elementary submodel of $H(\aleph_3)[\Gamma_{P_{\omega_2}}]$.

Consider any $\mu \leq \zeta < \omega_2$ and suppose
there is a condition $p\in P_{\omega_2}$ that forces
that $ U  \setminus \dot U_\zeta$ is not empty
for all non-empty $U \in \tau$. By possibly extending $p$
we may assume there is an integer $m$ 
such that $p\Vdash m\in \dot U_\zeta$.

   It follows by standard elementarity that, in $V[G_\mu]$,
   $\tau$ is a Fr\'echet-Urysohn topology on $\omega$.
  Let $\mathcal I$  denote the family of $\tau$-converging sequences. 
  Since $M[G_\mu]$ is closed under $\omega$-sequences each
   $I\in \mathcal I$ is an element of $M[G_\mu]$. Therefore,
  by elementarity, every $I\in\mathcal I$ is a converging
  sequence with respect to the final topology.

  Now observe that $p$ forces that $\dot U_\zeta$
  is a $(\tau,\mathcal I)$-nowhere dense set and
  that $m\in \dot U_\zeta$. Let $p,\tau,\mathcal I$ and 
  $\dot U_\zeta$ be
  elements of a countable elementary submodel $\tilde M$
  of $H(\aleph_3)[G_\mu]$. Note that each
   $(\tau,\mathcal I)$-nowhere dense set $Y$ in $\tilde M$
   is truly a $\tau$-nowhere dense set.  
   It was proven in \cite{Barman} that there is a 
   sequence $A\in \mathcal I$ that converges to $m$
   and satisfies that $A\cap Y$ is finite for
   every one of the countably many $\tau$-nowhere
   dense sets $Y$ in $\tilde M$. In particular, 
   $A$ is $(\tilde M,\tau,\mathcal I)$-bounding.
   By Lemma \ref{Milleriterate}, there is a condition $q< p$
   that forces that $A$ is $(\tilde M[G_{\omega_2}],\tau,
   \mathcal I)$-bounding. It follows that
    $A\setminus \val_{G_{\omega_2}}(\dot U_\zeta)$ is infinite,
    and contradicts the fact that $A$ converges to $m$
    with respect to the final topology which contains
     $\val_G(\dot U_\zeta)$ as a clopen set.
\end{proof}

\egroup

\section{Fr\'echet-Urysohn spaces in the random real model}

In this section we prove that in the standard random real model,
every infinite $\kappa\leq \mathfrak c$ equals the $\pi$-weight
of some countable Fr\'echet-Urysohn space. The construction
of these spaces are variations of the construction in
 \cite{DowFrechet}.

\bigskip

Our base space is $\omega^{<\omega}$ considered with the tree ordering,
and we fix an order-preserving   enumeration
 $\{ t_\ell : \ell\in\omega\}$ of
 it. 
   Order-preserving   means
 that $t_n  < t_\ell  $ implies
 that $n<\ell$ and
 that $t_n  = t^\frown k  $
 and $t_\ell = t^\frown (k+1)$ implies
   $n<\ell$. 
 For $t\in \omega^{<\omega}$, 
 $[t]$ denotes the
 set of all $s$ that equal or extend $t$.
\bigskip

 Let $\tau_0$ denote the topology on $\omega^{<\omega}$
 using the elements of $\{ [t] : t\in \omega^{<\omega}\}$
 and their  complements  as a clopen subbasis. 
 For any $t,\sigma \in \omega^{<\omega}$, 
 we let $t ^\frown \sigma$ denote the function with domain
  $\mathop{dom}(t)+\dom(\sigma)$  satisfying 
  that $t \subset t ^\frown \sigma$ and
  for $(t^\frown \sigma)(|t|+i) = \sigma(i)$ for $i\in \dom(\sigma)$. 

\bigskip

For any index set $I$, let $\mathcal M_I$ denote
the measure algebra (a Boolean algebra) on $\omega^I$. That is,
 for each $i\in I$ and $m\in\omega$, 
  the clopen set $[( i, m)]= \{ \rho\in \omega^I : 
   \rho(i) = m\}
  $ has measure $\frac{1}{2^{m+1}}$. We will
  let $\mu_I$ denote the measure on the elements
  of $\mathcal M_I$.
If $I$ is countable, then $\mathcal M_I$ is
the set of  Borel subsets of $\omega^I$
modulo the sets of measure 0. $\mathcal M_I$
is (forcing) ordered by $b_1 \leq  b_2$ providing
 $\mu(b_1\setminus b_2) = 0$. 
For any $B \subset I$, let
  $\pi_B$ denote
 the projection map from $\omega^I$ onto $\omega^{  B}$. 
 Then, for any uncountable set $I$,
    $\mathcal M_I$ is equal to 
    the family 
      $\bigcup \{ \pi_{B}^{-1}(b) : b\in \mathcal M_B,
       \ \ B\in [I]^{\aleph_0}\}$.
Also for any countable $B\subset I$ and $b\in \mathcal M_B$,
 we can say that the support of $\pi_B^{-1}(b)$ is
 contained in $B$ and the value of 
  $\mu_I(\pi_B^{-1}(b))$ is equal to 
    $\mu_B(b)$.  
    It is easily checked
    that $\mu_I(\pi_B^{-1}(b))$ is well-defined.
    Finally, for $b_1,b_2\in \mathcal M_I$,
    we again say that $b_1 \leq b_2$ providing
     $\mu_I(b_1\setminus b_2) = 0$.
     Naturally when we are forcing with $\mathcal M_I$
     we mean that we are forcing with the poset
     of non-zero elements. 

 \medskip

Before proceeding with the construction of our examples,
 we complete the proof using the results 
from  \cite{Goldstern}, and other fundamental references, 
that every $\gamma$-set in such
 random reals models has cardinality at most $\aleph_1$. 

\begin{proposition}[CH] In any forcing\label{randomgamma} extension by 
adding random reals, every $\gamma$-set has cardinality
at most $\aleph_1$.
\end{proposition}

\begin{proof}
It was shown in \cite{Gerlits}, where $\gamma$-sets were introduced,
that every $\gamma$-set has the property of Rothberger denoted by   $C''$.
Rothberger \cite{Rothberger} 
showed that every set of reals with property
 $C''$ has strong measure zero. Following the notation
 of \cite{Goldstern}, a set $X\subset 2^\omega$ has strong
 measure zero, if for every $h\in \omega^\omega$,
  there is a function $\nu^h : \omega\mapsto 2^{<\omega}$
  satisfying that $\nu^h(k)\in 2^{h(k)}$ for all $k\in\omega$,
  such that $X\subset \bigcup_{k\in\omega}   [\nu^h(k)]$.

  Now, suppose that $\dot X$ is an $\mathcal M_\kappa$-name 
  of a set of strong measure zero for some cardinal $\kappa>\omega_1$.
  Let $G$ be an $\mathcal M_\kappa$-generic filter.
  Choose any elementary submodel $M$ of $H(\kappa^+)$ with $\kappa,\dot X\in M$,
   $M^\omega\subset M$, and $|M|=\aleph_1$.  Let $\dot X\cap M = \dot X_M$
   and let $G_M= G\cap \mathcal M_{\kappa\cap M}$. By standard results,
    also explained in more detail below, $X_M = \val_{G_M}(\dot X_M)$
    is a strong measure zero set in $V[G_M]$
    and $X_M$ is a subset of $\val_{G}(\dot X)$. 
 Since $\mathcal M_\kappa$ is an $\omega^\omega$-bounding poset,
  the family $\mathcal H = \omega^\omega\cap V$ is a dominating
  family in $V[G_M]$.  In addition, the model $M[G]$ is an elementary
  submodel of $H(\kappa^+)[G]$, while $M[G]\cap 2^\omega
   = V[G_M]\cap 2^\omega$.
    For each $h\in \mathcal H$, choose
  a function $\nu^h\in \Pi_{k\in\omega} 2^{h(k)}$ satisfying
  that $X_M\subset \bigcup_{k\in\omega} [\nu^h(k)]$. 
  By elementarity, for each $x\in \val_G(\dot X)$
  and each $h\in \mathcal H$, $x\in \bigcup_{k\in\omega}[\nu^h(k)]$.
  
  The final model $V[G]$ can be viewed
  as a forcing extension of $V[G_M]$ by the poset
   $\mathcal M_{\kappa\setminus M}$. It is shown
   in \cite{Goldstern}*{Fact 1.16} that $\mathcal M_{\kappa\setminus M}$
    is strongly $\omega^\omega$-bounding \cite{Goldstern}*{Definition 1.13}
    and therefore, by \cite{Goldstern}*{Lemma 2.30}, 
     $X\subset \bigcap_{h\in\mathcal H}\left(\bigcup_{k\in\omega}
      [\nu^h(k)]\right)$ is a subset of $V[G_M]\cap 2^\omega$.
\end{proof}

 Our chosen  index set for $\mathcal M_I$ is
$I = \kappa \times \omega^{<\omega} \times \omega$.
For each $\alpha<\kappa$ and 
each $t\in \omega^{<\omega}$, we
are adding a function,
 $ \dot g_{\alpha,t }  \in \omega^\omega$ 
 defined by the condition that the 
 clopen set $[(~(\alpha,t,n)~,m)]
 \in \mathcal M_I$ forces that $\dot g_{\alpha,t}(n) = m$.

    As usual, for a forcing statement $\varphi$,
    $[[ \varphi]]$ denotes
    the element of $\mathcal M_I$  
  that is equivalent to   $\varphi$ being forced to hold. 
    Take any maximal anti-chain of conditions
    that decide the truth value of
     $\varphi$ and set $[[\varphi]]$
     to be (forcing equivalent to)
      the join of all the conditions
      in the anti-chain that force
      $\varphi$ to hold. 
      Thus, e.g., $[[\dot g_{\alpha,t}(n)=m ]]$
      is the clopen set 
       $ [~((\alpha,t,n)\ ,m)~]$.

For each $B\subset \kappa$, 
let $\tilde B = B\times\omega^{<\omega}\times\omega$.
The mapping
   sending $b\in \mathcal M_{\tilde B}$ to $\pi_{\tilde B}^{-1}(b)$
   defines a complete embedding of $\mathcal M_{\tilde B}$
   into $\mathcal M_I=\mathcal M_{\tilde\kappa}$.  For clarity we
   may sometimes use $[[\varphi ]]_{\tilde \kappa}$ and 
   $[[\varphi ]]_{\tilde B}$ 
   to indicate which poset, $\mathcal M_{\tilde \kappa}$ or
    $\mathcal M_{\tilde B}$, is being referenced.
\medskip

If $\dot A$ is an $\mathcal M_I$-name of a subset
of $\omega^{<\omega}$, we will say that a  
set $B\subset \kappa$ contains the support of $\dot A$ providing
 $[[ t\in \dot A]]\in \{ \pi_{\tilde B}^{-1}(b) : b\in 
 \mathcal M_{\tilde B}\}$
 for every $t\in \omega^{<\omega}$. Similarly, if $b\in \mathcal M_I$,
  we may say that $B$ contains the support of $b$ if
  $b \in \{ \pi_B^{-1}(c) : c\in \mathcal M_{\tilde B}\}$.  
 Since $\mathcal M_I$
 is ccc, for every such $\dot A$  there is a countable
 set $B\subset \kappa$ such that $  B$  contains its support. 
 When $B\subset \kappa$ and  
 contains the support of an $\mathcal M_I$-name
 $\dot A$, 
 we let $\pi_{\tilde B}(\dot A)$ denote the $\mathcal M_{\tilde B}$-name
 where $[[t\in \pi_{\tilde B}(\dot A)]]_{\tilde B} = 
 [[ t\in \dot A]]_{\tilde \kappa}$.

 \medskip

Similarly,   
if $\dot A$ is an $\mathcal M_{\tilde B}$-name for a subset
of $\omega^{<\omega}$ with $B\subset\kappa$,  we will
use $\pi_{\tilde B}^{-1}(\dot A)$ to denote the $\mathcal M_I$-name
where $[[ t\in \pi_{\tilde B}^{-1}(\dot A)]]_I =
[[ t\in \pi_{\tilde B}^{-1}(\dot A)]]_{\tilde \kappa} 
 = \pi_{\tilde B}^{-1}([[m\in \dot A]]_{\tilde B})$.
\bigskip

Now what is the topology on 
 $\omega^{<\omega}$ in the forcing extension? 
\medskip

 For each $\alpha<\kappa$, we define    an $\mathcal M_{I}$-name,
  $\dot W_\alpha$, of a
dense  open subset $W_\alpha$ of the usual rational topology on $\omega^{<\omega}$. 
The support of $\dot W_\alpha$ will be $
\tilde{\{\alpha\}} = \{\alpha\}\times \omega^{<\omega}\times \omega$.
We will let $\dot U_\alpha$ denote the $\mathcal M_{I}$-name 
of the complement, $\omega^{<\omega}\setminus \dot W_\alpha$, of $\dot W_\alpha$.
Clearly $\tilde{\{\alpha\}}$ contains the support of $\dot W_\alpha$
and $\dot U_\alpha$.
\medskip

For each $\alpha\in\kappa$ and $n\in\omega$, 
 let $\dot s_{\alpha,t_n}$ be the $\mathcal M_I$-name for
 an element of $\omega^{<\omega}$ given by
   $t_n^\frown \left(\dot g_{\alpha,t_n}\restriction n+2\right)$. 
Observe that, for any $\sigma   \in \omega^{n+2}$,
 the forcing element $[[ \dot s_{\alpha,t_n}\subset
 t_n^\frown \sigma   ]]$
 has measure less than $\frac{1}{2^{n+2}}$.

   Then $\dot W_\alpha$ is the $\mathcal M_I$-name
   of the  set    
    $\bigcup \{ [ \dot s_{\alpha,  t_n }] : n\in\omega\}$,
    and, as mentioned above, 
      $\dot U_\alpha$ is the $\mathcal M_I$-name
      for $\omega^{<\omega}\setminus \dot W_\alpha$.
     It is worth remarking that for any $s\in \omega^{<\omega}$,
      the value of $[[ s\in \dot W_\alpha]]$ is
      equal to the join of the finite family
         $\{
         [[ \dot s_{\alpha, t_n}\subset s ]] : t_n\subset s\}$.
 \bigskip

For any $B\subset \kappa$, let 
 $\dot \tau_B$ denote the $\mathcal M_I$-name for 
 the topology on $\omega^{<\omega}$ generated by
 the family $\tau_0\cup \{ \dot U_\alpha : \alpha\in B\}$.
 For any countable $B\subset\kappa$, we will abuse
 notation and also let $\dot \tau_B$ denote
 the $\mathcal M_{\tilde B }$-name for the topology
 on $\omega^{<\omega}$ generated by the
 family $\tau_0\cup \{ \pi_B(\dot U_\alpha) : \alpha \in B\}$.

 \bgroup

 \def\proofname{Proof of Claim.~}

  \bigskip

\begin{claim} For any $n_0,\ell\in\omega$ and 
$t_{n_0}^\frown \ell\subset s\in \omega^{<\omega}$,
 \ $[[ t_{n_0}\in \dot U_\alpha ]]\wedge 
[[s\in \dot W_\alpha]]$ has measure less than $\frac{1}{2^{n_0+\ell}}$. 
For emphasis we  note that\label{one}
$[[t_{n_0}\in \dot U_\alpha]]$ and 
$[[s\in \dot W_\alpha]]$ have support $\tilde{\{\alpha\}}$.
\end{claim}

 \begin{proof} Choose the appropriate $\bar\ell\in\omega$, 
 so that there is a sequence $n_0<n_1<n_2<  \cdots < n_{\bar \ell}$ 
 satisfying that $$\{ t_{n_0},t_{n_1},t_{n_2},\ldots,
  t_{n_{\bar\ell}}\} = \{ t \in \omega^{<\omega}:
    t_0\subseteq t\subseteq s\}~.$$
    Note that $[[ \dot s_ {\alpha,t_{n_{\bar \ell}}}\subset s]]  = 0$.
  More generally, for any $1\leq i<\bar\ell$ such that
   $\mathop{dom}(s) \subset \mathop{dom}(t_{n_i})+n_i $,
     $[[ \dot s_{\alpha,t_{n_i}}\subset s]]$ is also $0$. 
For other values of  $1\leq i <  \bar \ell$,
 $[[ \dot s(\alpha, t_{n_{i}}) \subset s]]$ has measure 
 less than $\frac{1}{2^{n_i+1}}\leq \frac{1}{2^{n_0+\ell+i}}$.
 Since $[[ t\in \dot U_\alpha]]\wedge [[ s\in \dot W_\alpha]]$
 is equal to $[[t\in\dot U_\alpha]]\wedge
   \left(\vee_{1\leq i\leq \bar\ell} [[ \dot s_{ \alpha,t_{n_i}}\subseteq s]]\right)$,
   this proves the Claim.
 \end{proof}

\begin{claim} Let\label{two} 
  $n_0\in \omega $ and let countable $B\subset \kappa$.
For any $\mathcal M_{\tilde B}$-names  $\dot A$ and $\{\dot s_n : n\in \dot A\}$
such that $1\in \mathcal M_{\tilde B}$ forces that
$\dot A\subset \omega$ and that, for $n\in \dot A$,  
   $t_{n_0}^\frown n \subset \dot s_n $ 
  (i.e. that $\{ \dot s_n : n\in \dot A\}$
  is a sequence   
  that converges to $t_{n_0}$ with respect to    $\tau_0$).
  Then, for any $\alpha\in \kappa\setminus B$,   $1 \in \mathcal M_I$ forces
  that  if $t_{n_0}\in \dot U_\alpha$, 
  then $\{ \pi_{\tilde B}^{-1}(\dot s_n) : n\in \pi_{\tilde B}^{-1}(\dot A)\}$ 
   is  mod finite contained in $\dot U_\alpha$.
    \end{claim}

\def\val{\operatorname{val}}

    \begin{proof}  Let $G_B$ be an $\mathcal M_{\tilde B}$-generic filter
    and work in the forcing extension $V[G_B]$. 
Let $A = \val_{G_B}(\dot A)$ and, for each $n\in A$,
 let $s_n = \val_{G_B}(\dot s_n)$. For each $n\in\omega\setminus
 A$, let $s_n = t_{n_0}^\frown n$. It follows that $\{ s_n : n\in\omega\}$
 converges to $t_{n_0}$ with respect to $\tau_0$. In this forcing extension,
  the final model is obtained by forcing with $\mathcal M_{I\setminus B}$,
  and let $b = [[ t_{n_0}\in \dot U_\alpha]]$ be the Boolean value 
  with respect to $\mathcal M_{I\setminus \tilde B}$. 
   Consider any $b_1$  less than $b$.  Choose any $\ell_0\in\omega$
   such that $\sum_{\ell>\ell_0} \frac{1}{2^{n_0+\ell}} =
     \frac{1}{2^{n_0+\ell_0} }< \mu(b_1)$. 
     By Claim \ref{one}, $[[ t_{n_0}\in \dot U_\alpha ]]\wedge
     [[ s_\ell\in \dot W_\alpha ]]$ has measure
     less than $\frac{1}{2^{n_0+\ell}}$ for each $\ell >\ell_0$. 
     Since all these elements of $\mathcal M_I$ have support
     contained in $ \{\alpha\}$ and $\alpha\notin B$,
      it follows that their projections into $\mathcal M_{I\setminus
      \tilde B}$ have the same measure. 
Therefore,      
   $$
   [[ t_{n_0}\in \dot U_\alpha ]]_{I\setminus\tilde B}   
   \wedge 
   [[~(\exists \ell>\ell_0)~~ s_\ell\in \dot W_\alpha]]_{I\setminus \tilde B}
   \leq    
   [[ t_{n_0}\in \dot U_\alpha ]]_{I\setminus \tilde B}
   \wedge 
   (\bigvee_{\ell>\ell_0} [[ s_\ell\in \dot W_\alpha ]]_{I\setminus\tilde B})   
   $$
   has measure
   less than $\frac{1}{2^{n_0+\ell_0}}<\mu(b_1)$. 
    This proves that $b_1$ does not force, with respect to $\mathcal M_{I\setminus
    \tilde B}$, that 
     $\dot W_\alpha\cap \{ s_n : n\in\omega\}$ is infinite. 
     Since $b_1$ was arbitrary, this proves that $[[t_{n_0}\in
     \dot U_\alpha]]_{I\setminus \tilde B}$ forces
     that $\{ s_n : n\in A \} $ is mod finite contained
     in $\dot U_\alpha$ as claimed.
    \end{proof}

\begin{claim}  For each $\lambda \leq \kappa$, $\dot \tau_\lambda$
is forced by $1$ over\label{three} $\mathcal M_I$ to be Fr\'echet-Urysohn.
\end{claim}

\begin{proof}  Let $n_0\in \omega$ and let
 $\dot Y$ be any $\mathcal M_I$-name of a subset of $\omega^{<\omega}$
 and assume that $b\in \mathcal M_I$ forces that $t_{n_0}$ is
 in the $\dot \tau_{\lambda}$-closure of $\dot Y$. Choose
 any countable $B\subset \kappa$ such that the supports
of $b$ and  of $\dot Y$ are contained in $B$. 
Let $b\in G$ for any $\mathcal M_I$-generic filter $G$
and let $G_B = \{ c\in \mathcal M_{\tilde B} : 
\pi_B^{-1}(c)\in G\}$. It follows that 
 $G_B$ is an
$\mathcal M_{\tilde B}$-generic filter and
that $V[G_B]$ is a submodel of $V[G]$ and
that $V[G]$ is an $\mathcal M_{I\setminus \tilde B}$-generic
extension of $V[G_B]$.

 We work in the model $V[G_B]$.
 Note  that  $\val_{G}(\dot Y)$
is equal to $\val_{G_B}(\pi_B(\dot Y))$. Let
 $Y = \val_{G_B}(\pi_B(\dot Y))$ and let $\tau_B$ equal
  $\val_{G_B}(\dot \tau_B)$.  
Note that $t_{n_0}$ is in the $\tau_B$-closure of $Y$. 
Since $\tau_B\supset \tau_0$ has a countable basis, 
it is Fr\'echet-Urysohn,
 and so we may choose a countable $A\subset\omega$ 
 and  a sequence $\{ s_n : n\in A\}\subset Y$
 that $\tau_B$-converges to $t_{n_0}$ and 
 such that, for each $n\in A$, $t_{n_0}^\frown n\subset s_n$.

 Now jump back to $V[G]$ and consider any $\tau_I$-basic open set
 $U$ such that $t_{n_0}\in U$,
  where $\tau_I = \val_{G}(\dot \tau_I)$.
  By the definition of this topology, there is a finite
  set $\{ \alpha_1,\ldots, \alpha_m\}\subset \kappa\setminus B$
  and an element $W\in \tau_B$ such that
  $$t_{n_0}\in W\cap \val_{G}(\dot U_{\alpha_1})
  \cap \cdots \cap
   \val_{G}(\dot U_{\alpha_m})\subset U~.$$
   Since, for each $1\leq i\leq m$,
     $[[ t_{n_0}\in \dot U_{\alpha_i}]]\in G$, 
     it follows from Claim \ref{two}, that
      a cofinite subset of $\{ s_n : n\in A\}$
      is contained in $U$. 
\end{proof}

 We come now to our final claim.

 \begin{claim} For each $\gamma < \kappa$, $1$ forces\label{cfour}
 that $\dot U_\gamma$ does not contain any non-empty
 element of $\dot \tau_\gamma$. In other words, for
 any uncountable cardinal $\lambda\leq \kappa$, $\dot \tau_\lambda$
 has $\pi$-weight equal to $\lambda$.
 \end{claim}

\begin{proof}  Let $G$ be any generic filter for
 $\mathcal M_I$ and let $G_{\{\gamma\}} = 
 G\cap \mathcal M_{\tilde{\{\gamma\}}}$.  We argue
 from the ground model $V[G_{\{\gamma\}}]$ using
 the $\tau_0$-dense open set $W_\gamma = \val_{G}(\dot W_\gamma)$. 
 For each $n_0\in\omega$, we have the sequence
  $\{ s(t_{n_0},\ell) : \ell\in \omega\}\subset W_\gamma$
  where
   $s(t_{n_0},\ell)  = 
   \val_{G_{\{\gamma\}}}(\pi_{\tilde{\{\gamma\}}}(
   \dot s_{\gamma, t_{n_0}^\frown\ell}))$.
   Since, for each $\ell\in\omega$,
    $t_{n_0}^\frown \ell \subset s(t_{n_0},\ell)$,
     the sequence $\{ s(t_{n_0},\ell) : \ell\in \omega\}$
     converges with respect to $\tau_0$ to $t_{n_0}$. 
    It then follows from Claim \ref{two}, that 
    in the forcing extension $V[G_{\tilde{\{\gamma\}}}][
     G_{I\setminus \tilde{\{\gamma\}}}]$,
      that, for each $n_0\in\omega$,
       $\{ s(t_{n_0},\ell) : \ell \in \omega\}$
       also converges to $t_{n_0}$ with respect
       to the topology $\val_{G}(\dot \tau_\gamma)$. 
    This, of course, implies that for each $n_0\in\omega$,
     $t_{n_0}$ is not in the $\val_{G}(\dot \tau_\gamma)$-interior
     of $\val_G(\dot U_\gamma)$.      
\end{proof}
 \egroup
 
\begin{theorem} For any infinite cardinal $\kappa$,  in the forcing
extension by the random real poset $\mathcal M_\kappa$, 
the $\pi$-weight spectrum for countable regular Fr\'echet-Urysohn
spaces contains all uncountable cardinals that 
are less than or equal to $\kappa$.
\end{theorem}
 
\begin{proof}  The forcing poset $\mathcal M_\kappa$ is 
isomorphic to the poset $\mathcal M_I$ where $I = \kappa\times
 \omega^{<\omega}\times \omega$. Let, for $\alpha<\kappa$,
 the $\mathcal M_I$-name $\dot U_\alpha$ be defined as above.
  Now let $G$ be an $\mathcal M_I$-generic filter,
  and for each $\alpha < \kappa$, let
   $U_\alpha = \val_{G}(\dot U_\alpha)$. 

  For each cardinal $\omega_1\leq \lambda \leq \kappa$,
  let $ \tau_\lambda$ denote the topology generated
  by using the family of
 finite intersections from
   the set $\tau_0\cup \{  U_\alpha : \alpha < \lambda\}$
    as a (clopen) base for a topology.
 It follows from Claim \ref{three} that $\tau_\lambda$
   is Fr\'echet-Urysohn. 

 It follows from Claim \ref{cfour}, that, for all
  $\gamma < \lambda$, the family
  of finite intersections from the set
   $\tau_0\cup \{  U_\alpha : \alpha < \gamma\}$ 
   does not form a $\pi$-base for $\tau_\lambda$.
   This proves that for
   every regular $\omega_1\leq\lambda\leq\kappa$
   the topology $\tau_\lambda$ is Fr\'echet-Urysohn
   with $\pi$-weight $\lambda$. 
   
Now we prove the same for any singular cardinal
   $\mu\leq\kappa$.  
  Consider any $S\subset \mu$ with $|S|=\lambda<\mu$.
 Choose any   bijection $f$ in the ground
  model from 
   $\mu$ to $\mu $
  that sends $S$ to the initial segment $\lambda$. 
  It then follows that 
 there is a homeomorphism
   $H_f$
   from $(\omega^{<\omega},\tau_\mu)$ to
    $(\omega^{<\omega},\tau_{\mu})$ that
    satisfies $H_f(U_\alpha) = U_{f(\alpha)}$
    for all $\alpha< \mu$.
This is proven using that $f$  can be used to induce
an automorphism on $\mathcal M_I$ that lifts 
canonically to $\mathcal M_I$-names that would
send $\dot W_\alpha$ to $\dot W_{f(\alpha)}$
for all $\alpha <\mu$ and $\dot W_\beta $ to 
$\dot W_\beta$ for all $\mu\leq \beta<\kappa$.
By Claim \ref{three}, $U_\lambda $ does not contain
any non-empty element of $\tau_\lambda$,
and so, by invoking the homeomorphism $H_f^{-1}$,
we have that the family of  clopen sets generated
by the family
$\tau_0\cup \{ U_{\alpha} : \alpha \in S\}$
is not a $\pi$-base for $\tau_\mu$.
\end{proof} 

 The following are the two most basic open problems
 about the $\pi$-weight spectrum of the countable regular
  Fr\'echet-Urysohn spaces. The first was  astutely  asked by J. Moore
  during a talk by the author at the 2012 Summer Topology conference.
  
\begin{question} Is there a countable regular Fr\'echet-Urysohn space
with $\pi$-weight equal to $\mathfrak b$?
\end{question}

\begin{question} Is it consistent that there are uncountable cardinals,
 $\lambda_1 < \lambda_2 < \lambda_3$ such that 
  $\lambda_2$ is not an element, while $\lambda_1, \lambda_3$
  are elements of the $\pi$-weight spectrum of the countable
   regular Fr\'echet-Urysohn spaces?
 \end{question}


\begin{bibdiv}

\def\cprime{$'$} 

\begin{biblist}
\bib{Abraham}{article}{
   author={Abraham, Uri},
   title={Proper forcing},
   conference={
      title={Handbook of set theory. Vols. 1, 2, 3},
   },
   book={
      publisher={Springer, Dordrecht},
   },
   isbn={978-1-4020-4843-2},
   date={2010},
   pages={333--394},
   review={\MR{2768684}},
   doi={10.1007/978-1-4020-5764-9\_6},
}


\bib{Barman}{article}{
   author={Barman, Doyel},
   author={Dow, Alan},
   title={Selective separability and ${\rm SS}^+$},
   journal={Topology Proc.},
   volume={37},
   date={2011},
   pages={181--204},
   issn={0146-4124},
   review={\MR{2678950}},
}
 

\bib{BaDo2}{article}{
   author={Barman, Doyel},
   author={Dow, Alan},
   title={Proper forcing axiom and selective separability},
   journal={Topology Appl.},
   volume={159},
   date={2012},
   number={3},
   pages={806--813},
   issn={0166-8641},
   review={\MR{2868880}},
   doi={10.1016/j.topol.2011.11.048},
}

 

\bib{DowFrechet}{article}{
  author={Dow, Alan},
   title={$\pi$-weight and the Fr\'echet-Urysohn property},
   journal={Topology Appl.},
   volume={174},
   date={2014},
   pages={56--61},
   issn={0166-8641},
   review={\MR{3231610}},
   doi={10.1016/j.topol.2014.06.013},
}

\bib{Gerlits}{article}{
   author={Gerlits, J.},
   author={Nagy, Zs.},
   title={Some properties of $C(X)$. I},
   journal={Topology Appl.},
   volume={14},
   date={1982},
   number={2},
   pages={151--161},
   issn={0166-8641},
   review={\MR{0667661}},
   doi={10.1016/0166-8641(82)90065-7},
}
\bib{Goldstern}{article}{
   author={Goldstern, Martin},
   author={Judah, Haim},
   author={Shelah, Saharon},
   title={Strong measure zero sets without Cohen reals},
   journal={J. Symbolic Logic},
   volume={58},
   date={1993},
   number={4},
   pages={1323--1341},
   issn={0022-4812},
   review={\MR{1253925}},
   doi={10.2307/2275146},
}

\bib{haberMiller}{misc}{
      author={Haberl, Valentin},
      author={Szewczak, Piotr},
      author={Zdomskyy, Lyubomyr},
       title={Concentrated sets and $\gamma$-sets in the Miller model},
        date={2024},
         url={https://arxiv.org/abs/2310.03864},
}


\bib{Hrusak-Ramos}{article}{
   author={Hru\v s\'ak, M.},
   author={Ramos-Garc\'ia, U. A.},
   title={Malykhin's problem},
   journal={Adv. Math.},
   volume={262},
   date={2014},
   pages={193--212},
   issn={0001-8708},
   review={\MR{3228427}},
   doi={10.1016/j.aim.2014.05.009},
}

\bib{RZmsep}{article}{
   author={Repov\v s, Du\v san},
   author={Zdomskyy, Lyubomyr},
   title={$M$-separable spaces of functions are productive in the Miller
   model},
   journal={Ann. Pure Appl. Logic},
   volume={171},
   date={2020},
   number={7},
   pages={102806, 8},
   issn={0168-0072},
   review={\MR{4099834}},
   doi={10.1016/j.apal.2020.102806},
}


\bib{Rothberger}{article}{
   author={Rothberger, F.}, 
   title={ Eine Versch{\"u}rfung der Eigenschaft},
   journal={Fund. Math.},
   volume={30},
   date={1938}, 
   pages={50-55} 
}


 \end{biblist}
\end{bibdiv}
\end{document}